\documentclass[11pt]{amsart}
\usepackage{amssymb, amsthm, amscd, epsfig, graphicx,
enumerate, stmaryrd, amsmath, graphs, booktabs, hhline, xcolor, mathtools, float} 

\newtheorem{thm}{Theorem}[section]
\newtheorem{cor}[thm]{Corollary}
\newtheorem{lem}[thm]{Lemma}
\newtheorem{prop}[thm]{Proposition}

\theoremstyle{definition}
\newtheorem{defn}[thm]{Definition}

\theoremstyle{remark}

\numberwithin{equation}{section}

\newcommand{\al}{\alpha}
\newcommand{\be}{\beta}
\newcommand{\ga}{\gamma}

\newcommand{\de}{\delta}

\newcommand{\la}{\lambda}

\newcommand{\va}{\varphi}

\newcommand{\x}{\times}

\newcommand{\R}{\mathbb R}

\newcommand{\del}{\partial}

\newcommand{\co}{\colon\thinspace}

\begin{document}
\mathsurround=1pt 
\title{Some approximate solutions to  the spread of influenza A virus infection}




\author{Boldizs\'ar Kalm\'{a}r}


\email{boldizsar.kalmar@gmail.com}

\begin{abstract}
We study differential equations describing the spread of influenza A virus infection based on a mathematical model.
 We look for surface trajectories of the dynamical system in hand and their asymptotic behaviour. 
\end{abstract}

\maketitle


\section{Introduction}

We consider the system of differential equations
\begin{equation}\label{egyenlet0}
\begin{array}{lll}
  \del_t  T(x, t)  & = &   -\be T(x, t) V(x, t) \\
   \del_t E_1(x, t)  & = &  \be T(x, t) V(x, t) - \frac{n_E}{\tau_E} E_1(x, t) \\
   \del_t E_i(x, t) & = &  \frac{n_E}{\tau_E} E_{i-1}(x, t) -  \frac{n_E}{\tau_E} E_{i}(x, t) \\
   \del_t I_1(x, t) & = &  \frac{n_E}{\tau_E} E_{n_E}(x, t) - \frac{n_I}{\tau_I} I_1(x, t) \\
   \del_t I_j(x, t)  & =  &  \frac{n_I}{\tau_I} I_{j-1}(x, t) -  \frac{n_I}{\tau_I} I_{j}(x, t) \\
   \del_t V(x, t)  &  =   &   p\sum_{j = 1}^{n_I}  I_j (x, t) - cV(x, t) + D_{PCF} a + v_a W(x, t) \\
   \del_x V(x, t) & = &  W(x, t)
 \end{array}
\end{equation}
where $i = 2, \ldots, n_E$ and $j = 2, \ldots, n_I$.
We replace the value $\del_{xx} V(x, t)$ in the model \cite{BQRY20} 
by the constant $a$, so in this paper 
we  deal with only the case of constant $\del_{xx} V(x, t)$. 
  The constants 
$D_{PCF}$ and $v_a$ correspond to 
the diffusion rate and to the upward transport
in the periciliary fluid. In the equations 
$c$ denotes the viral clearance rate,
$p$ denotes the virion producing rate and $T$ denotes the fraction of uninfected target cells.
We denote by $n_I$ and $n_E$ the numbers of age classes of infectious and eclipse phases, see \cite{BQRY20}. 

In the present paper, in the case of a system of equations with $n_E = 0$, 
we propose a heuristic which shows  that around 
$T(x, t) = \frac{c}{\tau_I p \be}$ the system changes its behaviour. 
We obtain that if $\tau_I p \be T > c$, then 
the trajectories tend to have asymptotic behaviour while for 
$\tau_I p \be T < c$ they do not. 
Denote by 
$c_E$ and $c_I$ the values $\frac{n_E}{\tau_E}$ and $\frac{n_I}{\tau_I}$.

To solve the equations, 
we are looking for the smooth functions
$$T, E_i, I_j, V, W \co \R^2 \to \R.$$ 
This means we are looking for 
 a map
$$
S \co 
\left[
\begin{array}{ccc}
  x    \\
   t  
\end{array}
\right]
 \mapsto 
 \left[
 \begin{array}{ccc}
 T(x, t)    \\
 E_1(x, t)   \\
 \vdots \\
  E_{n_E}(x, t)   \\
I_1(x, t) \\
\vdots \\
 I_{n_I}(x, t)   \\
 V(x, t) \\
 W(x, t)
\end{array}
\right]\mbox{, \ \ \ \ \ }
S' (x, t) = 
\left[
\begin{array}{ccc}
 \del_x T(x, t)  & \del_t T(x, t)  \\
 \del_x E_1(x, t)  &   \del_t E_1(x, t)   \\
 \vdots & \vdots \\
  \del_x E_{n_E}(x, t)  &   \del_t E_{n_E}(x, t)  \\
 \del_x I_1(x, t)  &   \del_t I_1(x, t)   \\
 \vdots & \vdots \\
  \del_x I_{n_I}(x, t)  &   \del_t I_{n_I}(x, t)  \\
 \del_x V(x, t) & \del_t V(x, t) \\
  \del_x W(x, t) & \del_t W(x, t)
\end{array}
\right]
$$
such that at every $(x_0, t_0)$
the graph of the linear map
$
(x, t) \mapsto S' (x_0, t_0) \cdot
 \left[
\begin{array}{ccc}
  x    \\
   t 
\end{array}
\right]
$
is the same as 
the 
graph of the linear map
$$
\Phi_{T, E_1, \ldots, E_{n_E},  I_1, \ldots, I_{n_I},  V, W} \co 
 \left[
\begin{array}{ccc}
  x    \\
   t 
\end{array}
\right] \mapsto 
\left[
\begin{array}{ccc}
\al &   -\be TV    \\
 \ga_0 &   \be TV - c_E E_1    \\
\ga_1 &  c_E E_{1} - c_E E_{2}    \\
 \vdots & \vdots \\ 
 \ga_{n_E-1}  &  c_E E_{n_E-1} - c_E E_{n_E}    \\
\de_0  &   c_E E_{n_E} - c_I I_1   \\
\de_1 &    c_I I_{1} - c_I I_{2}   \\
    \vdots & \vdots \\
\de_{n_I-1}  &    c_I I_{n_I-1} - c_I I_{n_I}   \\
 W  &    p\sum_{j = 1}^{n_I}  I_j  - cV + D_{PCF} a + v_a W  \\
    a  &    \psi 
\end{array}
\right] \cdot 
 \left[
\begin{array}{ccc}
  x    \\
   t 
\end{array}
\right]
$$
for all the points $$(T, E_1, \ldots, E_{n_E},  I_1, \ldots, I_{n_I},  V, W ) \in \R^{n_E + n_I + 3},$$ where
$$
\al, \ga_0, \ga_1, \ldots, \ga_{n_E-1},  \de_0, \ldots, \de_{n_I-1}, \psi  \co (T, E_1, \ldots, E_{n_E},  I_1, \ldots, I_{n_I},  V, W)  \to \R$$
are
some given smooth functions. 

If the linear maps 
$$\Phi_{T, E_1, \ldots, E_{n_E},  I_1, \ldots, I_{n_I},  V, W}$$
have rank two, then their graphs 
form a $2$-distribution 
in $\R^{n_E + n_I + 3}$ and we are looking for integral surfaces (graphs of the solutions $S$) of this distribution.

\section{Preliminaries}

%
%
Let us partition the space  $\R^{n_E + n_I + 3}$
according to different values of
the first $T$ coordinate 
into hyperplanes.
Then for each $T \in \R$
consider
the  map
$$
\va \co 
\left[
\begin{array}{ccc}
  E_1 \\
   \vdots \\ 
   \vdots \\
   E_{n_E} \\ 
    I_1 \\ 
    \vdots \\ 
    \vdots \\
    I_{n_I} \\  
    V \\ 
    W
 \end{array}
\right]  \mapsto   
\left[
\begin{array}{ccccc|cccccc}
  -c_E   &  0 & \cdots  &  \cdots & 0 & 0 & \cdots & \cdots & 0 &  \be T & 0  \\
c_E   &  -c_E & 0  & \cdots &   0 & 0 & \cdots & \cdots & 0 &  0 & 0  \\
\vdots & & & & & & & & &  & \vdots  \\
0   &  \cdots & 0 & c_E   &  -c_E  & 0 & \cdots &  \cdots & \cdots & 0 & 0  \\
0   &  \cdots & 0  &   0 & c_E & -c_I &  0 & \cdots & \cdots & \cdots & 0  \\
0   &  \cdots & 0  &   0 & 0 & c_I &  -c_I & 0 & \cdots &\cdots & 0  \\
\vdots & & & & \vdots  & & & & &  & \vdots  \\
0   &  \cdots & \cdots  &   \cdots & 0 & \cdots & 0 & c_I &  -c_I & 0  & 0 \\
0   &  \cdots & 0  &   0 & \cdots & p & \cdots &  \cdots & p & -c  & v_a \\
  0   &  \cdots & \cdots  &   \cdots & \cdots & \cdots &  \cdots & \cdots & \cdots & \cdots & 0
 \end{array}
\right] \cdot 
\left[
\begin{array}{ccc}
 E_1 \\
   \vdots \\ 
   \vdots \\
   E_{n_E} \\ 
   \hline
    I_1 \\ 
    \vdots \\
    \vdots \\  
    I_{n_I} \\  
    V \\ 
    W
\end{array}
\right]
$$
defined by the multiplication with this   coefficient matrix   $A \in \R^{(n_E + n_I + 2) \x (n_E + n_I + 2) }$. 
 Then the second column in $\R^{(n_E + n_I + 3)}$ of the coefficient matrix of the linear map 
$$
\Phi_{T, E_1, \ldots, E_{n_E},  I_1, \ldots, I_{n_I},  V, W}
$$
can be obtained by extending the vector
$$
 \va (E_1, \ldots, E_{n_E},  I_1, \ldots, I_{n_I},  V, W) + (0, \ldots, 0, D_{PCF} a, \psi)
$$
by the first coordinate $-\be TV$.


For example, the first column of the matrix of the map 
$
\Phi_{T, E_1, \ldots, E_{n_E},  I_1, \ldots, I_{n_I},  V, W}
$
can be chosen to be 
$$
[ -r_1 W, r_2 W, \ldots, r_{n_E + n_I + 2} W, a],
$$
where $r_i > 0$  are constants and $r_{n_E + n_I + 2} = 1$.

These two column vectors  define two vector fields in $\R^{(n_E + n_I + 3)}$. If these are linearly independent, then 
we have a $2$-distribution and if this distribution is involutive, then there are integral surfaces for it. These surfaces are the trajectories
of the differential equation.

\begin{lem}
We have 
$$
 (-1)^{n_E + n_I}  \det(A - \la I) = 
 (  c_E  + \la)^{n_E} (c_I + \la)^{n_I}(c + \la)\la + \be T c_E^{n_E}p \left(c_I^{n_I} - (c_I + \la)^{n_I} \right)
 $$
 for the characteristic polynomial of $A$.
\end{lem}
\begin{proof}
We compute the determinant of 
$$
\left[
\begin{array}{ccccc|cccccc}
  -c_E  - \la &  0 & \cdots  &  \cdots & 0 & 0 & \cdots & \cdots & 0 &  \be T & 0  \\
c_E   &  -c_E - \la & 0  & \cdots &   0 & 0 & \cdots & \cdots & 0 &  0 & 0  \\
\vdots & & & & & & & & &  & \vdots  \\
0   &  \cdots & 0 & c_E   &  -c_E - \la & 0 & \cdots &  \cdots & \cdots & 0 & 0  \\
0   &  \cdots & 0  &   0 & c_E & -c_I -\la &  0 & \cdots & \cdots & \cdots & 0  \\
0   &  \cdots & 0  &   0 & 0 & c_I &  -c_I -\la & 0 & \cdots &\cdots & 0  \\
\vdots & & & & \vdots  & & & & &  & \vdots  \\
0   &  \cdots & \cdots  &   \cdots & 0 & \cdots & 0 & c_I &  -c_I -\la & 0  & 0 \\
0   &  \cdots & 0  &   0 & \cdots & p & \cdots &  \cdots & p  & -c -\la  & v_a \\
  0   &  \cdots & \cdots  &   \cdots & \cdots & \cdots &  \cdots & \cdots & \cdots & 0 & -\la
 \end{array}
\right] 
$$
It is equal to 
$$
(  -c_E  - \la)^{n_E} (-c_I - \la)^{n_I}(-c - \la)(-\la)
+ (-1)^{n_E + n_I}\be T c_E^{n_E} \det B,
$$
where
$$B =
\left[
\begin{array}{cccccc}
 c_I &  -c_I -\la & 0 & \cdots  & 0  \\
 0 & c_I & -c_I -\la & 0  & 0 \\
 \vdots & & & &   \vdots  \\
 0 & \cdots & c_I &  -c_I -\la   & 0 \\
 p & \cdots &  \cdots & p & v_a \\
  0 &  \cdots & \cdots & 0 & -\la
 \end{array}
\right] 
$$
and 
$$
\det B = 
(-\la) \det B_{n_I}\mbox{, \ \ \ \ }
B_{k} =  \left[
\begin{array}{cccccc}
 c_I &  -c_I -\la & 0 & \cdots    \\
 0 & c_I & -c_I -\la & 0   \\
 \vdots & & &   \\
 0 & \cdots & c_I &  -c_I -\la    \\
 p & \cdots &  \cdots & p   
 \end{array}
\right]
$$
where $3 \leq k \leq n_I$. 
We have 
$$
\det B_{k} = c_I \det B_{k-1} + (-1)^{k + 1}p  (-c_I - \la)^{k - 1}
$$
for $k \geq 3$ and
$$
\det B_2 = \det 
\left[
\begin{array}{cccccc}
  c_I &  -c_I -\la    \\
  p & p 
 \end{array}
\right] = c_I p   - p (-c_I - \la). 
$$
Hence 
\begin{multline*}
\det B = (-\la) \left(c_I\det B_{n_I-1} + (-1)^{n_I + 1}p  (-c_I - \la)^{n_I - 1}\right) = \\
(-\la) \left( c_I \left( c_I \det B_{n_I-2} + (-1)^{n_I}p  (-c_I - \la)^{n_I - 2}\right) + (-1)^{n_I + 1}p  (-c_I - \la)^{n_I - 1}\right) = \\
(-\la) \left( c_I^2 \det B_{n_I-2} + c_I (-1)^{n_I}p  (-c_I - \la)^{n_I - 2} +  (-1)^{n_I + 1}p  (-c_I - \la)^{n_I - 1} \right) = \\
(-\la) (  c_I^3 \det B_{n_I-3} + c_I^2 (-1)^{n_I - 1}p  (-c_I - \la)^{n_I - 3}+ c_I (-1)^{n_I}p  (-c_I - \la)^{n_I - 2} + \\ (-1)^{n_I + 1}p  (-c_I - \la)^{n_I - 1} ) =
\cdots  = \\ (-\la)\left(  c_I^{n_I- 2} \det B_2 +  c_I^{n_I- 3}  (-1)^{4}p  (-c_I - \la)^{2}  + \cdots +  (-1)^{n_I + 1}p  (-c_I - \la)^{n_I - 1}   \right) = \\
 (-\la)\left(  c_I^{n_I- 1} p  + c_I^{n_I- 2} p  (c_I  +\la) +  c_I^{n_I- 3} p  (c_I + \la)^{2}  + \cdots + p  (c_I + \la)^{n_I - 1}   \right).
\end{multline*}

This means that 
\begin{multline*}
(-1)^{n_E + n_I}  \det(A - \la I) = 
(  c_E  + \la)^{n_E} (c_I + \la)^{n_I}(c + \la)\la
+ \\ \be T c_E^{n_E}  p(-\la)(  c_I^{n_I- 1}  + c_I^{n_I- 2}  (c_I  +\la) +  c_I^{n_I- 3}  (c_I + \la)^{2}  + \cdots \\+  (c_I + \la)^{n_I - 1}   )= \\ 
 (  c_E  + \la)^{n_E} (c_I + \la)^{n_I}(c + \la)\la + \be T c_E^{n_E}p (-\la) \sum_{j = 0}^{n_I - 1} c_I^j (c_I + \la)^{n_I- 1 - j}.
\end{multline*}
\end{proof}

%

\begin{cor}\label{kov2}
If $n_E = 0$, then
$$
(-1)^{n_I}  \det(A - \la I) = 
(c_I + \la)^{n_I}(c + \la)\la + \be T p \left(c_I^{n_I} - (c_I + \la)^{n_I} \right).
 $$
\end{cor}

\section{Results}

Recall that $c$ denotes the viral clearance rate,
$p$ denotes the virion producing rate and $T$ denotes the fraction of uninfected target cells.
We denote by $n_I$ and $n_E$ the numbers of age classes of infectious and eclipse phases and we have $c_I = n_I / \tau_I$
and $c_E = n_E / \tau_E$. We want to determine the eigenvalues of the matrix $A$.

\subsection{The case of $n_E =0$}

If we suppose that $n_E = 0$, then 
by Corollary~\ref{kov2} we have 
\begin{multline*}
(-1)^{n_I}  \det(A - \la I) =  (c_I + \la)^{n_I}(c + \la)\la + \be T p \left(c_I^{n_I} - (c_I + \la)^{n_I} \right) = \\
 (c_I + \la)^{n_I} \left( (c + \la)\la - \be T p \right) + \be T p c_I^{n_I} = 
 (c_I + \la)^{n_I} \left( \la^2 + c\la - \be T p \right) + \be T p c_I^{n_I},
\end{multline*}
denote $(c_I + \la)^{n_I} \left( \la^2 + c\la - \be T p \right)$ by $F(\la)$.  Of course $F(0) < 0$ so 
$0$ is between the  roots 
$$\frac{-c \pm \sqrt{ c^2 + 4\be T p }}{2}$$
of $\la^2 + c\la - \be T p$.
This implies for example that

The derivative of $(-1)^{n_I}  \det(A - \la I)$ is the same as the derivative of $F(\la)$, which is 
\begin{multline*}
n_I (c_I + \la)^{n_I -1} (c + \la)\la + (c_I + \la)^{n_I} \la + (c_I + \la)^{n_I} (c+\la) - \be T p n_I (c_I + \la)^{n_I -1} = \\
(c_I + \la)^{n_I -1} \left(     n_I (c + \la)\la +    (c_I + \la)(c + 2\la) - \be T p n_I    \right) = \\
(c_I + \la)^{n_I -1} \left(   (n_I + 2)\la^2 + (cn_I + c + 2c_I)\la + cc_I - \be T p n_I     \right).
\end{multline*}
Observe that 
$$(-c + \sqrt{ c^2 + 4\be T p })/2 \leq -c_I$$
is impossible since
$$
0 < c_I \leq (c - \sqrt{ c^2 + 4\be T p })/2 < 0$$
is impossible.

\subsubsection{The case of $-c_I  \leq   (-c - \sqrt{ c^2 + 4\be T p })/2$}

Suppose that 
$$
-c_I  \leq  (-c - \sqrt{ c^2 + 4\be T p })/2
$$
so $F(\la)$ has  two or three different roots, more simply
$$
c_I \geq (c + \sqrt{ c^2 + 4\be T p })/2$$
$$
2(c_I - c/2) \geq  \sqrt{ c^2 + 4\be T p },
$$
which is equivalent to 
$$
4(c_I - c/2)^2 \geq  c^2 + 4\be T p$$
because 
$
2(c_I - c/2) \leq  -  \sqrt{ c^2 + 4\be T p }
$
is not possible, so 
$$
c_I^2 - cc_I \geq  \be T p,$$
and note that this also implies that 
$$
c_I > c.$$

If $n_I$ is even, then 
$$(c_I + \la)^{n_I} \left( (c + \la)\la - \be T p \right) + \be T p c_I^{n_I}$$
has a horizontal tangency 
at 
$$
\la_a = -c_I$$
and a local minimum at 
$$
\la_b = \frac{ - c(n_I + 1) - 2c_I + \sqrt{   \left( c(n_I + 1) + 2c_I \right)^2 - 4(n_I + 2) ( cc_I - \be T p n_I) }   }{2(n_I + 2)}
$$
because 
$$\lim_{\pm \infty} F = \infty.$$
Since
$F(0) = - \be T p c_I^{n_I}$,
we have  $F(\la_2) < 0$ so it also has
 at most one  singular point  between $\la_a$ and $\la_b$.
We have that 
$$
(-1)^{n_I}  \det(A - \la I)  = F(\la) + \be T p c_I^{n_I}$$
is equal to $0$ at $\la = 0$. 
These imply that 
if $$cc_I = \be T p n_I,$$ then 
the root of $\det(A - \la I)$
is only $\la_0 = 0$, if 
$$cc_I < \be T p n_I,$$ then 
the roots are $\la_0 = 0$ and some $\la_1 > 0$ and 
if $$cc_I > \be T p n_I,$$
then 
the roots are $\la_0 = 0$ and some $-c_I < \la_1 < 0$.

If $n_I \geq 3$ is odd, then 
$\la = -c_I$ is just an inflection point and then we have the analogous result about the roots
but in every case there is another root $\la_2 < -c_I$ as well. 

\subsubsection{The case of $(-c - \sqrt{ c^2 + 4\be T p })/2 < -c_I  <  (-c + \sqrt{ c^2 + 4\be T p })/2$}

In the case of 
$$
(-c - \sqrt{ c^2 + 4\be T p })/2 < -c_I  <  (-c + \sqrt{ c^2 + 4\be T p })/2
$$
the function $F(\la)$ has again three different roots, more simply
$$
(c + \sqrt{ c^2 + 4\be T p })/2 > c_I  >  (c - \sqrt{ c^2 + 4\be T p })/2
$$
$$
 \sqrt{ c^2 + 4\be T p } > 2c_I  - c >   - \sqrt{ c^2 + 4\be T p }
$$
$$
| 2c_I - c | < \sqrt{ c^2 + 4\be T p }$$
$$
4c_I^2 - 4cc_I < 4\be T p$$
$$
c_I^2 - cc_I < \be T p.$$
If $n_I$ is even, then 
$(c_I + \la)^{n_I} \left( (c + \la)\la - \be T p \right) + \be T p c_I^{n_I}$
has two local minima
at
$$
\la_{1, 2} = \frac{ - c(n_I + 1) - 2c_I \pm  \sqrt{   \left( c(n_I + 1) + 2c_I \right)^2 - 4(n_I + 2) ( cc_I - \be T p n_I) }   }{2(n_I + 2)}
$$
with $F(\la_{1, 2})<0$ and 
a local maximum at $\la = - c_I$, where $\la_1 < -c_I < \la_2$. 
So 
$\det(A - \la I)$ has the roots $\la_0 = 0$ and  $\la_1 > 0$ if $cc_I < \be T p n_I$, 
$\la_0 = 0$ if $cc_I = \be T p n_I$ and $\la_0 = 0$ and some $\la_1 < 0$ if $cc_I > \be T p n_I$
and 
 at most two other negative roots in each cases.
 If $n_I \geq 3$ is odd, then 
 similarly to the case of $2 | n_I$
 we have $\la_0 = 0$ and a positive or negative root depending on the same conditions
  and one other negative root.

This motivates 
the following definition.

\begin{defn}
If  $c > \tau_I p \be T$, then the  system (\ref{egyenlet0}) is \emph{definite}.
If $c < \tau_I p \be T$, then the  system is
\emph{indefinite}.
\end{defn}

If the system that we have studied is definite or indefinite, then $0$ is an eigenvalue with multiplicity one. 
So we obtain the following.

\begin{prop}
Let $n_E = 0$.
If the system (\ref{egyenlet0}) is definite, then all the non-zero eigenvalues of the matrix of $\va$ is negative.
If the system is indefinite, then there is one positive eigenvalue and all the other non-zero eigenvalues are negative.
In both cases  $0$ is an eigenvalue.
\end{prop}
\begin{proof}
The statement follows from the previous arguments. 
\end{proof}

\begin{table}[h]
\begin{center}
\begin{tabular}{|c|c|c|c|}
\hline
$c < \be T p \tau_I$ & $0, +$ & $0, +$ & $(-, -), 0, + $  \\
\hline
$c = \be T p \tau_I$ & $0$ & $0$ & $(-, -), 0$  \\
\hline
$c > \be T p \tau_I$ & $-, 0$ & $-, 0$ & $(-, -), -, 0$    \\
\bottomrule
  & $c_I^2 - cc_I > \be T p$ & $c_I^2 - cc_I = \be T p$ & $c_I^2 - cc_I < \be T p$   \\
\hline
\end{tabular}
\end{center}
\caption{The signs of the possible different eigenvalues of $A$ in the case of $2 | n_I$. The ``$-$'' signs in the brackets are optional and 
represent negative numbers which are smaller than $-c_I$. The other ``$-$'' signs represent negative numbers which are bigger than $-c_I$.}
\label{tablazat_1}
\end{table}

\begin{table}[h]
\begin{center}
\begin{tabular}{|c|c|c|c|}
\hline
$c < \be T p \tau_I$ & $-, 0, +$ & $-, 0, +$ & $-, 0, + $  \\
\hline
$c = \be T p \tau_I$ & $-, 0$ & $-, 0$ & $-, 0$  \\
\hline
$c > \be T p \tau_I$ & $-, -, 0$ & $-, -, 0$ & $-, -, 0$    \\
\bottomrule
  & $c_I^2 - cc_I > \be T p$ & $c_I^2 - cc_I = \be T p$ & $c_I^2 - cc_I < \be T p$   \\
\hline
\end{tabular}
\end{center}
\caption{The signs of the possible different  eigenvalues of $A$ in the case of odd $n_I$. 
 The ``$-$'' signs in the first and second  rows represent negative numbers which are smaller than $-c_I$
 and in the third row the ``$-$'' signs represent
 a negative number which is smaller than $-c_I$ and a 
 negative number which is bigger than $-c_I$.}
\label{tablazat_2}
\end{table}

\begin{prop}
The algebraic multiplicity 
of all the real eigenvalues is one, except possibly for the negative eigenvalues in the brackets in the case of $2 | n_I$ and for the eigenvalue $0$ 
in the case of $c = \be T p \tau_I$.
\end{prop}
\begin{proof}
The statement follows from the previous arguments. 
\end{proof}

\begin{prop}
The geometric  multiplicity 
of every real eigenvalue is one.
\end{prop}
\begin{proof}
We have to solve the equation
$$
(A - \la I)x  = 0,
$$
which is just 
$$
\begin{array}{ccc}
 (-c_I - \la)I_1 + \be TV = 0, \\
  c_I I_1 + (-c_I - \la)I_2 = 0, \\
  \ldots \\
  c_I I_{n_I - 1} + (-c_I - \la)I_{n_I} = 0, \\
  p\sum_{i=1}^{n_I} I_i + (-c - \la)V + v_a W = 0, \\
  -\la W = 0.
 \end{array}
$$
This implies $W = 0$ and since
substituting 
the first $n_I$ equations into $$p\sum_{i=1}^{n_I} I_i + (-c - \la)V = 0$$
results $\det ( A - \la I) = 0$, which can be seen after a short computation, the solution of  
$(A - \la I)x  = 0$ is 
$$
W = 0, \mbox{\ \ \ }V \in \R,  \mbox{\ \ \ } I_k = \frac{c_I^{k-1}}{(c_I + \la)^k}\be T V,$$ 
which means that the geometric multiplicity is equal to $1$ 
for $\la \neq 0$.
If $\la = 0$, then 
the solution is 
$$
V \in \R,  \mbox{\ \ \ } I_1 = \be T V /c_I,  \mbox{\ \ \ } I_k = I_{k-1},  \mbox{\ \ \ }
p\sum_{i=1}^{n_I} I_i -c V + v_a W= 0$$
so 
the geometric multiplicity is equal to $1$.
\end{proof}

\begin{cor}
The eigenvectors for $\la \neq 0$ are
$$
\left[   \frac{1}{c_I + \la}\be T V,  \frac{c_I}{(c_I + \la)^2}\be T V,  \ldots, \frac{c_I^{n_I-1}}{(c_I + \la)^{n_I}}\be T V,    V, 0 \right],
$$
where $V \in \R$ and for $\la = 0$ they are
$$
\left[   \frac{1}{c_I}\be T V,  \frac{1}{c_I}\be T V,  \ldots, \frac{1}{c_I} \be T V,    V,   \frac{\be Tp \tau_I - c}{-v}  V  \right],
$$
where $V \in \R$. 
\end{cor}

In the following, suppose that $n_E = 0$.  
According to the previous arguments by parametrizing everything by $T$ 
let us write and denote the space $\R^{n_I + 3}$ as
$$
\R \oplus  \R^{n_I + 2} =  \mathbb V \oplus \mathbb V^{-}_T \oplus \mathbb V^+_T \oplus \mathbb V^0_T,
$$
where the first $\mathbb V$ summand stands for the variable $T$, the summand $\mathbb V^-_T$ is the eigenspace for the negative eigenvalues,
 the summand $\mathbb V^+_T$ is the eigenspace for the positive eigenvalues
 and  the last 
$\mathbb V^0_T$ summand stands for
the  eigenvalue $0$.
Then we get the vector field
$$\Phi_{T,   I_1, \ldots, I_{n_I},  V, W} \in \R^{n_I + 3}$$
by 
computing at first the vector field 
$$
\va(  I_1, \ldots, I_{n_I},  V, W) \in \mathbb V^{-}_T \oplus \mathbb V^+_T \oplus \mathbb V^0_T$$
and then by taking the 
vector field
$$
\va(  I_1, \ldots, I_{n_I},  V, W) + 
(0, \ldots, 0, D_{PCF} a, 0 ) + (0, \ldots, 0, \psi).
$$
This last addition modifies 
$\va(  I_1, \ldots, I_{n_I},  V, W)$ only in its last two coordinates in $\R^{n_I + 2}$.
After this we add 
 $-\be TV$ in the first coordinate $\R$ in $\R \oplus \R^{n_I + 2}$.
 
 The map $\va$ has negative eigenvalues in the direction $\mathbb V^{-}_T$, an eigenvalue 
 $0$ in the direction of $\mathbb V^0_T$ 
  and a positive  eigenvalue  in the direction of $\mathbb V_T^+$ depending 
  on whether the state of the system is definite or indefinite (the presence of the positive eigenvalue changes at 
 $T = \frac{c}{\tau_I p \be}$).
The last coordinate of the vector field 
 $\va(  I_1, \ldots, I_{n_I},  V, W)$  is equal to $0$. 
 By omitting $\mathbb V^0_T$ 
 it is easy to sketch this vector field depending on the value of $T$,  see Figure~\ref{kuszob}.

\begin{figure}[h!]
\begin{center}
\epsfig{file=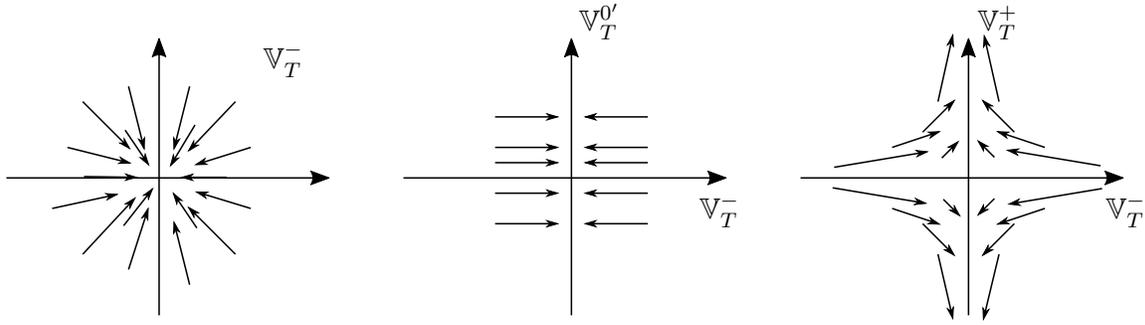, height=3.9cm}
\put(-11.5, 3.4){$\mathbb V^-_T$}
\put(-5.7, 1.4){$\mathbb V^-_T$}
\put(-0.3, 1.4){$\mathbb V^{-}_T$}
\put(-7.3, 3.9){$\mathbb V^{0'}_T$}
\put(-2, 3.9){$\mathbb V^+_T$}
\end{center} 
\caption{The sketch of the vector fields for $T < \frac{c}{\tau_I p \be}$, $T = \frac{c}{\tau_I p \be}$ and $T > \frac{c}{\tau_I p \be}$, respectively, 
from the left to the right hand side.
The subspace $\mathbb V^{0'}_T$ is a component of  $\mathbb V^0_T$.}
\label{kuszob}
\end{figure}

When we add $(0, \ldots, 0, D_{PCF} a, 0 )$, all the vector fields are modified by
the constant $D_{PCF} a$ in one coordinate  direction. 
%
%
%
Then in the space $\mathbb V^{ n_I} \oplus \mathbb V^2 = \R^{n_I + 2}$
we add $(0, \ldots, 0, \psi)$ to the previously constructed vector field.
If the value of $|\psi |= |\del_{xt} V(x, t)|$ is small, then 
adding it does not modify the picture too much. 

%
%
%
%

\end{document}